\documentclass[12pt]{article}
\usepackage{amsmath, amsthm}\usepackage{enumerate}\usepackage{amssymb}\usepackage{bbold}\usepackage{color}
\newtheorem{theorem}{Theorem}[section]
\newtheorem{proposition}[theorem]{Proposition}
\newtheorem{lemma}[theorem]{Lemma}
\newtheorem{example}[theorem]{Example}
\newtheorem{corollary}[theorem]{Corollary}
\newtheorem{remark}[theorem]{Remark}

\newtheorem{definition}[theorem]{Definition}

\DeclareMathOperator{\sto}{\xrightarrow[]{st-o}}

\DeclareMathOperator{\pc}{\xrightarrow[]{p}}
\DeclareMathOperator{\std}{\downarrow^{st}}
\DeclareMathOperator{\oc}{\xrightarrow[]{o}}

\DeclareMathOperator{\stpd}{\downarrow^{st_p}}
\DeclareMathOperator{\stpc}{\xrightarrow[]{st_p}}
\DeclareMathOperator{\stc}{\xrightarrow[]{\mu-st_\tau}}

\begin{document}

\title{Statistical $p$-convergence in lattice-normed Riesz spaces}
\maketitle
\author{\centering {Abdullah Ayd\i n$^{1,*}$, Reha Yapal\i$^{2}$, Erdal Korkmaz$^{3}$}\\ 
\small Department of Mathematics, Mu\c{s} Alparslan University, Mu\c{s}, 49250, Turkey, \\
$^{1}$a.aydin@alparslan.edu.tr, $^{2}$reha.yapali@alparslan.edu.tr, $^{3}$e.korkmaz@alparslan.edu.tr\\ \small $^*$Corresponding Author}
\begin{abstract}
A sequence $(x_n)$ in a lattice-normed space $(X,p,E)$ is statistical $p$-convergent to $x\in X$ if there exists a statistical $p$-decreasing sequence $q\stpd 0$ with an index set $K$ such that $\delta(K)=1$ and $p(x_{n_k}-x)\leq q_{n_k}$ for every $n_k\in K$. This convergence has been investigated recently for $(X,p,E)=(E,|\cdot|,E)$ under the name of statistical order convergence and under the name of statistical multiplicative order convergence, and also, for taking $E$ as a locally solid Riesz space under the names statistically unbounded $\tau$-convergence and statistically multiplicative convergence. In this paper, we study the general properties of statistical $p$-convergence.
\end{abstract} 


\maketitle

\section{Introduction}
Lattice-valued norms on Riesz spaces and statistical convergence of sequences provide natural and efficient tools in the theory of functional analysis. The main aim of the present paper is to combine these two subjects, and so, is to introduce the concept of {\em statistical $p$-convergence} on lattice-normed spaces and illustrate the usefulness of lattice-valued norms for the investigation of different types of statistical convergence in Riesz spaces, which attracted the attention of several authors in a series of recent papers \cite{Aydn1,Aydn2,AE,SP,TA}. 

Riesz space that was introduced by F. Riesz in \cite{Riez} is an ordered vector space having many applications in measure theory, operator theory, and applications in economics (cf. \cite{AB,ABPO,AlTo,LZ,Za}). On the other hand, lattice-normed space is a lattice valued norm, and also, it should be noticed that the theory of lattice-normed spaces is well developed in the case of decomposable lattice norms; see \cite{K,KK}. But, unless otherwise stated, we do not assume lattice norms to be decomposable. One study related to this paper is done by Ayd\i n et al., in \cite{AEEM1}, where unbounded order convergence is defined on lattice-normed vector lattice.

As an active area of research, statistical convergence is a generalization of the ordinary convergence of a real sequence, and the idea of statistical convergence was firstly introduced by Zygmund \cite{Zygmund}. After then, Fast \cite{Fast} and Steinhaus \cite{St} independently improved that idea. Several applications and generalizations of the statistical convergence of sequences have been investigated by several authors (cf. \cite{Con,Fast,Fridy,St,TA,Trip}). 

The concept of statistical convergence in Riesz spaces is introduced by Ercan \cite{Ercan}, where the notion of the statistically $u$-uniformly convergent sequence is introduced in Riesz spaces. Then \c{S}en\c{c}imen and Pehlivan extended the statistical convergence to Riesz space with respect to order convergence; see \cite{SP}. Recently, Ayd\i n et al., have investigated some studies about statistical convergence on Riesz spaces and locally solid Riesz spaces; see \cite{Aydn1,Aydn2,AE,TA}.

The structure of the paper is as follows. In Section 2, we give several notions related to lattice-normed spaces and to statistical convergence. In Section 3, we introduce the notions of statistical $p$-decreasing and statistical $p$-convergence sequence in lattice-normed spaces. In the last section, we study on main results.

\section{Preliminaries}
We begin the section with some basic concepts related to the theory of Riesz space and refer to \cite{AB,ABPO,K,KK,LZ,Za} for more details.
\begin{definition}
	A real-valued vector space $E$ with a partial order relation "$\leq$" on $E$ (i.e. it is an antisymmetric, reflexive and transitive relation) is called {\em ordered vector space} whenever, for every $x,y\in E$, we have
	\begin{enumerate}
		\item[(a)] $x\leq y$ implies $x+z\leq y+z$ for all $z\in E$,
		\item[(b)] $x\leq y$ implies $\lambda x\leq \lambda y$ for every $0\leq \lambda \in \mathbb{R}$.
	\end{enumerate}	
\end{definition}

An ordered vector space $E$ is called {\em Riesz space} or {\em vector lattice} if, for any two vectors $x,y\in E$, the infimum and the supremum
$$
x\wedge y=\inf\{x,y\} \ \ \text{and} \ \ x\vee y=\sup\{x,y\}
$$
exist in $E$, respectively. A Riesz space is called {\em Dedekind complete} if every nonempty bounded above subset has a supremum (or, equivalently, whenever every nonempty bounded below subset has an infimum). For an element $x$ in a Riesz spaces $E$, {\em the positive part}, {\em the negative part}, and {\em module} of $x$, respectively 
$$
x^+:=x\vee0, \ \ \ x^-:=(-x)\vee0\ \ and \ \ |x|:=x\vee (-x).
$$ 
In the present paper, the vertical bar $|\cdot|$ of elements of the Riesz spaces will stand for the module of the given elements. A Riesz space $E$ is called {\em Archimedean} whenever $\frac{1}{n}x\downarrow0$ holds in $E$ for each $x\in E_+$. Every Riesz space need not be Archimedean. To see this, we give the following example.
\begin{example}\label{archimedean}
	Consider the Riesz space $E:=\mathbb{R}^2$ with lexicographical ordering: $(x_1,x_2)\leq(y_1,y_2)$ if and only if $x_1<y_1$ or $x_1=y_1$ and $x_2\leq y_2$. Then $E$ is not Archimedean because, for the positive element $(1,1)\in E$, we have $\frac{1}{n}(1,1)\downarrow$, but $\frac{1}{n}(1,1)\neq0$.
\end{example}

Unless otherwise stated, we assume that all vector lattices are real and Archimedean. A sequence $(x_n)$ in a Riesz space $E$ is said to be {\em increasing} whenever $x_1 \leq x_2\leq\cdots$, and {\em decreasing} if $x_1 \geq x_2\geq \cdots$. Then we denote them by $x_n\uparrow$ and $x_n\downarrow$ respectively. Moreover, if $x_n\uparrow$ and $\sup x_n=x$ then we write $x_n\uparrow x$. Similarly, if $x_n\downarrow$ and $\inf x_n=x$ then we write $x_n\downarrow x$. Then we call that $(x_n)$ is increasing or decreasing as monotonic.

One of the fundamental and crucial concepts in the study of Riesz space is the order convergence (cf. \cite[Thm.16.2]{LZ}). 
\begin{definition}
	A sequence $(x_n)$ in a Riesz space $E$ is called {\em order convergent} to $x\in E$ $($$x_n\oc x$, for short$)$ whenever there exists a sequence $(q_n)\downarrow0$ in $E$ such that $|x_n-x|\le q_n$ for all $n\in \mathbb{N}$.
\end{definition}
We refer the reader for an exposition on order convergence to \cite{AS}.

\begin{definition}
	Let $X$ be a vector space and $E$ be a Riesz space. Then $p:X \to E_+$ is called an {\em $E$-valued vector norm} whenever it satisfies the following conditions:
	\begin{enumerate}
		\item $p(x)=0\Leftrightarrow x=0$;
		\item $p(\lambda x)=|\lambda|p(x)$ for all $\lambda\in\mathbb{R}$;
		\item $p(x+y)\leq p(x)+p(y)$ for all $x,y\in X$.
	\end{enumerate}
\end{definition}
Then the triple $(X,p,E)$ is called a {\em lattice-normed space}, abbreviated as $LNS$. 
If, in addition, $X$ is a Riesz spaces and the vector norm $p$ is monotone (i.e., $|x|\leq |y|\Rightarrow p(x)\leq p(y)$ holds for all $x,y\in X$) then the triple $(X,p,E)$ is called {\em lattice-normed vector lattice} or {\em lattice-normed Riesz space}. We abbreviate it as $LNRS$.
The lattice norm $p$ in $LNS$s $(X,p,E)$ is said to be {\em decomposable} if, for all $x\in X$ and
$e_1,e_2\in E_+$, from $p(x)=e_1+e_2$ it follows that there exist $x_1,x_2\in X$ such that $x=x_1+x_2$ and $p(x_k)=e_k$ for $k=1,2$. While dealing with $LNRS$s, we shall keep in mind also the following examples.
\begin{example}\label{ExLNVL_2}
	Let $X$ be a normed space with a norm $\|\cdot\|$. Then $(X,\|\cdot\|,{\mathbb R})$ is an $LNS$ . 
\end{example}

\begin{example}\label{ExLNVL_3}
	Let $X$ be a Riesz space. Then $(X,|\cdot|,X)$ is an $LNRS$ . 
\end{example}

\begin{definition}
	Let $(X,p,E)$ be an $LNS$. Then a sequence $(x_n)$ in $X$ is called {\em $p$-convergent} to $x\in X$ $($or, shortly, $x_n\pc x)$ whenever $p(x_n)\oc p(x)$ holds in $E$.
\end{definition}
An $LNRS$ is called {\em $op$-continuous} if $x_n\oc 0$ implies that $p(x_n)\oc0$; see \cite{AEEM1}. We refer the reader for more information on $LNS$s to \cite{E,KK,K}. 

Now, we recall some basic properties of the concepts related to statistical convergence. Consider a set $K$ of positive integers. Then the {\em natural density} of $K$ is defined by
$$
\delta(K):=\lim_{n\rightarrow \infty}\frac{1}{n}\left \vert \left \{  k\leq n:k\in
K\right \}  \right \vert,
$$
where the vertical bar of sets will stand for the cardinality of the given sets. We refer the reader to an exposition on the natural density of sets to \cite{Fast,Fridy}. In the same way, a sequence $x=(x_{k})$ is called \textit{statistical convergent} to $L$ provided that
$$
\lim_{n\rightarrow \infty}\frac{1}{n}\left \vert \left \{  k\leq n:\left \vert
x_{k}-L\right \vert \geq \varepsilon \right \}  \right \vert =0
$$
for each $\varepsilon>0$. Then it is written by $S-\lim x_{k}=L$. We take the following notions from \cite{SP}.
\begin{definition}
	Let $(x_n)$ be a sequence in a Riesz space $E$. Then $(x_n)$ is called
	\begin{enumerate}
		\item[(a)] {\em statistical order decreasing} to $0$ if there exists a set $K=\{n_1<n_2<\cdots\}\subseteq\mathbb{N}$ with $\delta(K)=1$ such that $(x_{n_k})$ is decreasing and $\inf\limits_{n_k\in K}(x_{n_k})=0$, and so, it is abbreviated as $x_n\std 0$,
		\item[(b)] {\em statistically order convergent} to $x\in E$ if there exists a sequence $q_n\std0$ with a set $K=\{n_1<n_2<\cdots\}\subseteq\mathbb{N}$ such that $\delta(K)=1$ and $|x_{n_k}-x|\leq q_{n_k}$ for every $n_k\in K$, and so, we write $x_n\sto x$.
	\end{enumerate}
\end{definition}

\section{Statistical $p$-convergence}
In this section, we introduce the statistical convergence on $LNS$s. Some notions and results of this section are direct analogies of well-known facts of the theory of normed lattices. The notion of statistical monotony was first introduced in \cite{Trip} for real sequences. We extend the notions of statistical monotony and statistical convergence to statistical $p$-version of them in $LNS$s.
\begin{definition}
	Let $(X,p,E)$ be an $LNS$ and $(x_n)$ be a sequence in $X$. Then $(x_n)$ is said to be
	\begin{enumerate}
		\item[(1)] {\em statistical $p$-decreasing} to $0$ if there exists a set $K=\{n_1<n_2<\cdots\}\subseteq\mathbb{N}$ such that $\delta(K)=1$ and $p(x_{n_k})\downarrow 0$ on $K$, and so, we abbreviate it as $x_n\stpd0$,
		\item[(2)] {\em statistical $p$-convergence} to $x$ if there exists a sequence $q_n\stpd0$ in $X$ with a set $K=\{n_1<n_2<\cdots\}\subseteq\mathbb{N}$ such that $\delta(K)=1$ and $p(x_{n_k}-x)\leq q_{n_k}$ for all $n_k\in K$, and so, we write $x_n\stpc x$. 
	\end{enumerate}
\end{definition}

Note that every $p$-convergent sequence is statistical $p$-convergent to its $p$-limit. However, the converse need not be true in general. To illustrate this, we give the following example.
\begin{example}
	Consider the Riesz space $E:=c$ of all convergent real sequences. Then define a norm on $E$ by $p(x):=|x|+(\lim_{n\to\infty}|x_n|)\cdot \mathbb{1}$
	for every $x\in E$, where $\mathbb{1}$ denotes the sequence identically equal to $1$. Thus, $(E,p,E)$ is an $LNRS$. Now, take a sequence $(x_n)$ in $E$ denoted by $x_n:=(x_k^n)=(x_1^n,x_2^n,\cdots)\in E$ such that
	$$
	x_k^n:=
	\begin{cases} 
		m, & k=m^3\ and \ k=n  
		\\ 0, & k=m^3  \ and \ k\neq n
		\\ \frac{1}{k}, & k\neq m^3 \ and \ k=n 
		\\ 0, & k\neq m^3  \ and \ k\neq n 
	\end{cases}
	$$
	for $n,k,m\in \mathbb{N}$. Next, take another sequence $q_n:=(q_k^n)$ in $E$ denoted by $q_k^n:=m$ whenever $k=m^3$ and $k=n$, and otherwise, denoted by $q_k^n:=\frac{1}{k}$. Then it is clear that $q_{n_k}\downarrow 0$ on the index set $J:=\{n_j\in \mathbb{N}:\forall m\in\mathbb{N}, \ n_j\neq m^3\}$. Now, take $x_{k_j}$ as $y_k^n$ whenever $k\neq m^3$ for all $m\in\mathbb{N}$. Also, we have 
	$$
	p(x_{n_j})=p(y_k^n)=|y_k^n|+(\lim_{n\to\infty}|y_k^n|)\cdot \mathbb{1}=y_k^n\leq q_{n_j}
	$$
	for all $n_j\in J$. Therefore, we get $x_n\stpc 0$ in $(E,p,E)$. On the other hand, $(x_n)$ is not $p$-bounded, and so, it is not $p$-convergent sequence.
\end{example}

\begin{remark}
	Consider an $LNRS$ $(E,|\cdot|,E)$ for an arbitrary Riesz space $E$. Then the following statements hold.
	\begin{enumerate}
		\item[(i)] Statistical order decreasing and statistical $p$-decreasing of sequences coincide.
		\item[(ii)] Statistically order convergence and statistical $p$-convergence agree.
	\end{enumerate}
\end{remark}

In the following examples, we give some relations between the classical statistical convergences and statistical $p$-convergences.
\begin{example}
	Take the $LNRS$ $(\mathbb{R},|\cdot|,\mathbb{R})$. Then the statistical $p$-convergence of sequence implies the statistical convergence in $\mathbb{R}$ because order convergence implies the convergence in $\mathbb{R}$. 
\end{example}

\begin{example}
	Let $(X,\lVert\cdot\rVert)$ be a normed space. Then statistical $p$-convergence of sequence implies statistical norm convergence in $LNS$ $(X,\lVert\cdot\rVert,\mathbb{R})$.
\end{example}

\begin{remark}\label{Remarks} \
	\begin{enumerate}	
		\item[(i)] Every statistical $p$-decreasing to $0$ sequence is statistical $p$-convergent to $0$ in $LNS$s.
		
		\item[(ii)] A decreasing to $0$ sequence is a statistical $p$-decreasing to $0$ in $op$-continuous $LNRS$s. However, it need not be true in general $LNRS$s.
		
		\item[(iii)] Every $p$-decreasing sequence (i.e., sequences such that $p(x_n)\downarrow0$) is statistical $p$-decreasing to $0$ in $LNS$s.
		
		\item[(iv)] Any order convergent sequence in $op$-continuous $LNRS$s $st_p$-converges to its order limit.
		
		\item[(v)] The statistical $p$-decreasing limit is linear.
		
		\item[(vi)] If $0\leq r_n\leq q_n$ holds for all $n$ and $q_n\stpd0$ then $r_n\stpd0$ in $LNRS$s.
	\end{enumerate}
\end{remark}

The convergences of the properties of Example \ref{monotone not conversely} need not be true in general. To see one case of them, we give the following example.
\begin{example}\label{monotone not conversely}
	Let $E$ be the Riesz space $c_0$ of all convergent to zero real sequences. Consider a sequence $(x_n)$ in $LNRS$ $(E,|\cdot|,E)$ denoted by $x_n:=(x_k^n)=(x_1^n,x_2^n,\cdots,x_k^n,\cdots)\in E$ such that
	$$
	x_k^n:=
	\begin{cases} 
		1, & k=n^3 \\
		\frac{1}{k}, & k\neq n^3
	\end{cases}
	$$
	for all $n,k\in\mathbb{N}$. Then it is clear that $x_n\stpd 0$. Observe that the whole sequence $(x_n)$ is not monotonic, but it is statistically $p$-monotonic.
\end{example}

\begin{proposition}\label{uniqness}
	The statistical $p$-limit of a sequence is uniquely determined.
\end{proposition}

\begin{proof}
	Assume that a sequence $(x_n)$ in an $LNS$ $(X,p,E)$ satisfies $x_n\stpc x$ and $x_n\stpc y$ for some different element $x,y\in X$. Then it follows that there are two sequences $q_n\stpd0$ and $r_n\stpd0$ with index sets $K,L\subseteq\mathbb{N}$ such that $\delta(K)=\delta(L)=1$ and 
	$$
	p(x_{n_k}-x)\leq q_{n_k} \ \ \text{and} \ \ p(x_{n_l}-y)\leq r_{n_l}
	$$
	holds for all $n_k\in K$ and $n_l\in L$. Now, consider the set $M:=K\cap L$. Then we have $\delta(M)=1$ and 
	$$
	p(x_{n_m}-x)\leq q_{n_m} \ \ \text{and} \ \ p(x_{n_m}-y)\leq r_{n_m}
	$$
	holds for all $n_m\in M$. Hence, by the inequality
	$$
	0\leq p(x-y)\leq p(x_{n_m}-x) + p(x_{n_m}-y) \leq q_{n_m}+r_{n_m},
	$$
	and by $(q_{n_m}+r_{n_m})\downarrow 0$ on $M$, we obtain the desired result, i.e., $x=y$.
\end{proof}

\begin{proposition}
	The statistical $p$-limit is linear.
\end{proposition}

\begin{proof}
	Suppose that $x_n\stpc x$ and $y_n\stpc y$ holds in an $LNS$ $(X,p,E)$. Then there exist sequences $q_n\stpd 0$ and $r_n\stpd 0$ such that
	$$
	\delta(\{n\in\mathbb{N}:p(x_n-x)\nleq q_n\})=0 $$
	and
	$$
	\delta(\{n\in\mathbb{N}:p(y_n-y)\nleq r_n\})=0
	$$
	hold. On the other hand, we observe the following inequality
	$$
	\{n\in\mathbb{N}:p(x_n+y_n-x-y)\nleq q_n+r_n\}\subseteq \{n\in\mathbb{N}:p(x_n-x)\nleq q_n\}\cup\{n\in\mathbb{N}:p(y_n-y)\nleq r_n\}.
	$$
	Hence, we have
	$$
	\delta\{n\in\mathbb{N}:p((x_n+y_n)-(x+y))\nleq q_n+r_n\}=0.
	$$
	It follows from $(q_n+r_n)^{st_p}\downarrow0$ that we get $x_n+y_n\stpc x+y$.
\end{proof}

The following result is a statistical $p$-version of the squeeze law.
\begin{proposition}
	Let $(X,p,E)$ be an $LNRS$ and $(x_n)$, $(y_n)$ and $(z_n)$ be sequences in $X$ such that $x_n\leq y_n\leq z_n$ for all $n\in\mathbb{N}$. Then $x_n\stpc x$ and $z_n\stpc x$ implies $y_n\stpc x$.
\end{proposition}

\begin{proof}
	Suppose that $x_n\stpc x$ and $z_n\stpc x$. Then we have some sequences $q_n\stpd 0$ and $r_n\stpd0$ with an index set $\delta(M)=1$ such that $p(x_{n_m}-x)\leq q_{n_m}$ and $p(z_{n_m}-x)\leq r_{n_m}$ for all $n_m\in M$. On the other hand, following from the inequality $x_n\leq y_n\leq z_n$ for all $n$, we have $y_n-x\leq z_n-x$ and $x-y_n\leq x-y_n$, and so, $p(y_n-x)\leq p(z_n-x)$ and $p(y_n-x)\leq p(x_n-x)$ for all $n\in \mathbb{N}$ because of the monotony of $p$. So, it follows that
	$$
	p(y_n-x)\leq q_{n_m}+r_{n_m}
	$$
	for all $n_m\in M$. Therefore, the proof is completed by $(q_n+r_n)\stpd0$, i.e., $y_n\stc x$.
\end{proof}

It is clear that a subsequence of a $p$-convergent sequence is also $p$-convergent to the $p$-limit of the sequence. However, for statistical $p$-convergence, it need not be true. To see this, we give the following example. 
\begin{example}
	Consider a sequence $(x_n)$ in $LNRS$ $(\mathbb{R},|\cdot|,\mathbb{R})$ denoted by $x_n:=n$ whenever $n=k^3$, and by $x_n=\frac{1}{n}$ whenever $n\neq k^3$ for some $k\in\mathbb{N}$. Then it is clear that $x_n\stpc 0$. But the subsequence $(x_{n_k})$, where $n_k:=m^3$ for some $m\in\mathbb{N}$, does not statistical $p$-convergent.
\end{example}

We finish this section with the following notions, which are direct analogies of well-known facts of the theory of normed lattices.
\begin{definition}
	Let $(x_n)$ be a sequence in an $LNS$ $(X,p,E)$ and $(x_n)$ be a sequence in $X$.
	\begin{enumerate}
		\item[(1)] $(x_n)$ said to be {\em statistically $p$-Cauchy} whenever there exists a sequence $q_n\stpd 0$ with a set $\delta(K)=1$ such that we have $p(x_{n_k}-x_{n_{k-1}})\leq q_{n_k}$ for all $n_k\in K$.
		\item[(2)] $(x_n)$ said to be {\em statistically $p$-bounded} if there exists a positive element $e\in E_+$ and an index set $\delta(K)=1$ such that $p(x_{n_k})\leq e$ for all $n_k\in K$.
		\item[(3)] An $LNS$ $(X,p,E)$ is called statistical $p$-complete if every statistical $p$-Cauchy sequence in $X$ is statistical $p$-convergent.
		\item[(4)] An $LNRS$ $(X,p,E)$ is called {\em statistical $p$-$K$B-space} if every statistically $p$-bounded and increasing sequence in $X_+$ is statistical $p$-convergent.
	\end{enumerate}	
\end{definition}

We remind that a Banach lattice is said to be a {\em $KB$-space} whenever every increasing norm bounded positive sequence is norm convergent. The following lemma is a standard fact for normed spaces. 
\begin{lemma}\label{subseq criterion}
	Let $(X,\lVert\cdot\rVert)$ be a normed space. Then $x_n\xrightarrow{\lVert\cdot\rVert} x$ iff for any subsequence $x_{n_k}$ there is a further subsequence $x_{n_{k_j}}$ such that $x_{n_{k_j}}\xrightarrow{\lVert\cdot\rVert} x$.
\end{lemma}

\begin{theorem}
	The following statements hold.
	\begin{enumerate}
		\item[(i)] Every statistical $p$-convergent sequence is statistical $p$-bounded.
		
		\item[(ii)] A statistical $p$-convergent sequence is statistical $p$-Cauchy. 
		
		\item[(iii)] Take a Banach lattice $(X,\lVert\cdot\rVert,\mathbb{R})$. Then it is a $KB$-space whenever it is statistical $p$-$KB$-space.
	\end{enumerate}
\end{theorem}

\begin{proof}
	$(i):$ Assume that $x_n\stpc x$ in an $LNS$ $(X,p,E)$. Then there exists a sequence $q_n\stpd0$ with a set $\delta(K)=1$ such that $p(x_{n_k}-x)\leq q_{n_k}$ for every $n_k\in K$. It follows from the $q_{n_k}\downarrow$ on $K$ that we have the following inequality 
	$$
	p(x_{n_k})\leq p(x_{n_k}-x)+p(x)\leq q_{n_k}+p(x)\leq q_{n_1}+p(x)
	$$
	for all $n_k\in K$. So, we obtain $p(x_{n_k})\leq q_{n_1}+p(x)\in E_+$ for each $n_k\in K$, i.e., $(x_n)$ is statistical $p$-bounded.
	
	$(ii)$ Suppose that $x_n\stpc x$ satisfies in an $LNS$, and so, there exists a sequence $q_n\stpd0$ with a set $\delta(K)=1$ such that $p(x_{n_k}-x)\leq q_{n_k}$ for each $n_k\in K$. Hence, for an arbitrary $n_k\in K$, we have
	$$
	p(x_{n_k}-x_{n_{k-1}})\leq p(x_{n_k}-x)+p(x_{n_{k-1}}-x)\leq q_{n_k}+q_{n_{k-1}}.
	$$
	Following from $(q_{n_k}+q_{n_{k-1}})\stpd0$, we get the desired result.
	
	$(iii):$ Assume that $0\leq x_n\uparrow$ and norm bounded sequence. Then, by taking the index set $K=\mathbb{N}$, $(x_n)$ is a statistical bounded sequence. Now, take an arbitrary subsequence $(x_{n_m})$ of $(x_n)$. Then $(x_{n_m})$ is also positive, increasing and statistical $p$-bounded sequence in $X$. Thus, $(x_{n_m})$ statistical $p$-converges to an element $x\in X$ because of the statistical $p$-$KB$-property of $X$. So, there exists a sequence $q_n\stpd0$ in $\mathbb{R}$ with an index set $\delta(J)=1$ such that $\lVert x_{n_{m_j}}-x\rVert\leq q_{n_{m_j}}$ for all $n_{m_j}\in J$. Therefore, we obtain $x_{n_{m_j}}\xrightarrow{\lVert\cdot\rVert} x$. It follows from Lemma \ref{subseq criterion} that $x_n\xrightarrow{\lVert\cdot\rVert} x$. As a result, $(X,\lVert\cdot\rVert,\mathbb{R})$ is a $KB$-space.
\end{proof}

\begin{proposition}
	A statistical $p$-$KB$-space $(X,\lVert\cdot\rVert,\mathbb{R})$ is statistical $p$-complete.
\end{proposition}

\begin{proof}
	Assume $(X,\lVert\cdot\rVert,\mathbb{R})$ is a statistical $p$-$KB$-space and $(x_n)$ is a statistical $p$-Cauchy sequence in $X$. Then there exists a sequence $q_n\stpd 0$ with a set $\delta(K)=1$ such that $p(x_{n_k}-x_{n_{k-1}})\leq q_{n_k}$ for all $n_k\in K$. Thus, $(x_{n_k})$ is a $p$-Cauchy. Hence, it follows from \cite[Exer.95.4]{Za} that $(x_{n_k})$ is $p$-convergent to some $x\in X$, and so, it is statistical $p$-convergent to $x$. Therefore, $(X,\lVert\cdot\rVert,\mathbb{R})$ is a statistical $p$-complete $LNS$.
\end{proof}

\section{Main Results}
We remind that if $(x_n)$ is a sequence satisfying the property $P$ for all $n\in \mathbb{N}$ except a set of natural density zero then we say that $(x_n)$ satisfies the property $P$ for almost all $n$, and it is abbreviated by a.a.n. \cite{Fridy}.
\begin{theorem}
	Let $(X,p,E)$ be an $LNS$ and $(x_n)$ be a sequence in $X$. Then $x_n\stpc x$ holds if and only if there is another sequence $(y_n)$ in $X$ such that $x_n=y_n$ for a.a.n and $y_n\stpc x$.
\end{theorem}

\begin{proof}
	Suppose that there exists a sequence $(y_n)$ in $X$ such that $x_n=y_n$ for a.a.n and $y_n\stpc x$. Then there exists a sequence $q_n\stpd0$ with a set $\delta(K)=1$ such that $p(y_{n_k}-x)\leq q_{n_k}$ for every $n_k\in K$. Since $x_n=y_n$ for a.a.n, we observe that the index sets $\{n_k\in K:p(x_{n_k}-x)\leq q_{n_k}\}$ and $\{n_k\in K:p(y_{n_k}-x)\leq q_{n_k}\}$ are a.a.n. Thus, there exists a subset $M\subseteq K$ such that $\delta(M)=1$ and $p(x_{n_m}-x)\leq q_{n_m}$ for every $n_m\in M$. Therefore, the desired result is to appear, i.e., we get $x_n\stpc x$.
\end{proof}
The lattice operations in $LNRS$s are statistically $p$-continuous
in the following sense.
\begin{theorem}\label{LO are continuous}
	If $x_n\stpc x$ and $y_n\stpc y$ satisfy then $x_n\vee y_n\stpc x\vee y$ and $x_n\wedge y_n\stpc x\wedge y$ in an $LNRS$ .
\end{theorem}

\begin{proof}
	Suppose that $x_n\stpc x$ and $y_n\stpc y$ holds in an $LNRS$ $(X,p,E)$. Then it is enough to show that $x_n\vee y_n\stpc x\vee y$ because the other case can be obtained by applying both \cite[Thm.11.5(iv)]{LZ} and the linearity of the statistical $p$-limit. Since $x_n\stpc x$ and $y_n\stpc y$, there are two sequences $q_n\stpd0$ and $r_n\stpd0$ with index sets $\delta(K)=\delta(L)=1$ such that
	$$
	p(x_{n_k}-x)\leq q_{n_k} \ \ \text{and} \ \ p(x_{n_l}-y)\leq r_{n_l}
	$$
	hold for all $n_k\in K$ and $n_l\in L$. It follows from the inequality $|a\vee b-a\vee c|\leq|b-c|$ (cf. \cite[Thm.1.9(2)]{ABPO}) in Riesz spaces that we have
	\begin{eqnarray*}
		p(x_n\vee y_n-x\vee y)&=&p(|x_n\vee y_n-x_n\vee y+x_n \vee y-x\vee y|)\\&\leq& p(|x_n\vee y_n-x_n\vee y|)+p(|x_n \vee y-x\vee y|)\\&\leq& p(|y_n-y|)+p(|x_n-x|)
	\end{eqnarray*}
	for all $n\in\mathbb{N}$. Thus, we obtain 
	$$
	p(x_{n_m}\vee y_{n_m}-x\vee y)\leq p(y_{n_m}-y)+p(x_{n_m}-x)\leq q_{n_m}+r_{n_m}
	$$
	for each $n_m\in M$, where $M:=K\cap L$. It follows from $(q_{n_km}+r_{n_m})\stpd 0$ that $x_n\vee y_n\stpc x\vee y$.
\end{proof}

\begin{corollary}
	If $x_n\stpc x$ satisfies in an $LNRS$ then the following statements hold
	\begin{enumerate}
		\item[(i)] $x_n^+\stpc x^+$,
		\item[(ii)] $x_n^-\stpc x^-$,
		\item[(iii)] $|x_n|\stpc |x|$.
	\end{enumerate}
\end{corollary}

\begin{proof}
	The straightforward proofs of $(i)$ and $(ii)$ are direct results of Theorem \ref{LO are continuous}. We show $(iii)$. Since $x_n\stpc x$, we have another sequence $q_n\stpd 0$ with an index set $\delta(K)=1$ such that $p(x_{n_k}-x)\leq q_{n_k}$ for all $n_k\in K$. Then, by using the inequality $||a|-|b||\leq|a|+|b|$ in Riesz spaces and by applying the monotony of $p$, we have
	$$
	p(|x_{n_k}|-|x|)=p(||x_{n_k}|-|x||)\leq p(|x_{n_k}-x|)\leq q_{n_k}
	$$
	for every $n_k\in K$. Therefore, we get $|x_n|\stpc |x|$.
\end{proof}

\begin{proposition}\label{positive}
	Let $(x_n)$ be a positive sequence in an $LNRS$. Then $x_n\stpc x$ implies $x\geq0$.
\end{proposition}

\begin{proof}
	Assume that $x_n\stpc x$ holds in an $LNRS$ $(X,p,E)$. Then it follows from Theorem \ref{LO are continuous} that 
	$$
	0\leq x_n=x_n\vee0\stpc x\vee0=x^+\in X_+.
	$$
	Thus, by considering Proposition \ref{uniqness}, we know that the statistical $p$-limit of a sequence is unique. Therefore, we obtain the desired result, $x=x_+\in X_+$.
\end{proof}

In the following theorem, we give a relation between statistical $p$-convergence and order convergence.
\begin{theorem}\label{implies order conv}
	Every monotone decreasing and statistical $p$-convergent sequence order converges to its statistical $p$-limit. 
\end{theorem}

\begin{proof}
	Suppose that $(x_n)$ is a monotone decreasing and statistical $p$-convergent to $x$ sequence in an $LNRS$ $(X,p,E)$. Then consider an arbitrary $m\in\mathbb{N}$. Thus, we have $x_m-x_n\in X_+$ for all $n\geq m$, and so, we get $x_m-x_n\stpc x_m-x\in X_+$ by Proposition \ref{positive}. Hence, $x_m\geq x$. Since $m$ is arbitrary, $x$ is a lower bound of $(x_n)$. Assume that $y$ is another lower bound of $(x_n)$, i.e., $x_n\geq y$ for all $n$. So, again by Proposition \ref{positive}, we have $x_n-y\stpc x-y\in X_+$, i.e., $x\geq y$. Therefore, we get the desired result, $x_n\downarrow x$.
\end{proof}

\begin{proposition}
Let $(X,p,E)$ be an $op$-continuous $LNRS$. Then the statistical order convergent sequence  in $X$ is statistical $p$-convergent.	
\end{proposition}

\begin{proof} 
Assume that $x_n\sto x$ in $X$. Then there exist a sequence $t_n\std0$ in $X$ with a set $\delta(M)=1$ such that $|x_{n_m}-x|\leq t_{n_m}$ for all $n_m\in M$. Thus, we obtain $p(x_{n_m}-x)\leq p(t_{n_m})$ for every $n_m\in M$. On the other hand, since $t_{n_m}\downarrow0$ on the set $M$ in $X$, we have $p(t_{n_m})\downarrow0$ on $M$ in $E$ because of the $op$-continuity of $(X,p,E)$. Thus, we have $t_n\stpd 0$, and so, we get $x_n\stpc x$.
\end{proof}

\begin{theorem}\label{$op$-cont-0}
	Let $(X,p,E)$ be an $LNRS$. Then the following statements are equivalent;
	\begin{enumerate}
		\item[(i)] $x_n\sto 0$ in $X$ implies $x_n\stpc 0$,
		\item[(ii)] $x_n\oc 0$ in $X$ implies $x_n\stpc 0$.
	\end{enumerate}	
\end{theorem}

\begin{proof} 
	$(i)\Rightarrow(ii)$: 
	Let $x_n\oc 0$ be a sequence in $X$. Then there exists a sequence $y_n\downarrow0$ in $X$ such that $|x_n|\le y_n$ for all $n\in \mathbb{N}$. So, we have $p(x_n)\leq p(y_n)$ for each $n$. On the other hand, we have $y_n\sto 0$ because of $y_n\downarrow0$. Now, by using $(i)$, we get $y_n\stpc 0$. Then there exists another sequence $q_n$ with an index set $\delta(K)=1$ such that $p(y_{n_k})\leq q_{n_k}$ for every $n_k\in K$. So, we obtain the following inequality
	$$
	p(x_{n_k})\leq p(y_{n_k})\leq q_{n_k}
	$$
	for all $n_k\in K$. Therefore, we have $x_n\stpc 0$.
	
	$(ii)\Rightarrow(i)$: It is straightforward.
\end{proof}

Recall that a subset $A$ of a Riesz space $E$ is called solid if, for each $x\in A$ and $y\in E$,  $|y|\leq|x|$ implies $y\in A$. Also, a solid vector subspace of a Riesz space is referred to as an ideal. Moreover, an order closed ideal is called a band (cf. \cite{AB,AlTo}).
\begin{proposition}
Let $(X,p,E)$ be an $LNRS$ and $B$ be a band in $X$. If $b_n\stpc x$ satisfies for a sequence $(b_n)$ in $B$ then we have $x\in B$.
\end{proposition}

\begin{proof}
Suppose that $b_n\stpc x$ holds for a sequence $(b_n)$ in $B$ and $x\in X$. Then it follows from Theorem \ref{LO are continuous} that we have $|b_n|\wedge|z| \stpc |x|\wedge|z|$ for any $z\in B^d:=\{z\in X: |z|\wedge|b|=0\ \text{for \ all} \ b\in B\}$. On the other hand, $|b_n|\wedge|z|=0$ for all $n$ because of $(b_n)$ in $B$. Hence, we have $|x|\wedge|z|=0$, and so, we obtain $x\in B^{dd}$. As a result, it follows from \cite[Thm.1.39]{ABPO} that we have $B=B^{dd}$, and so, we obtain $x\in B$.
\end{proof}

\begin{proposition}
Let $(X,p,E)$ be an $LNRS$. Take a projection band $B$ of $X$ and  the corresponding band projection $p_B$. Then $x_n\stpc x$ implies $P_B(x_n)\stpc P_B(x)$ in $X$. 
\end{proposition}

\begin{proof}
Assume that $x_n\stpc x$ holds in $(X,p,E)$. Then there exists a sequence $q_n\stpd 0$ with an index set $\delta(K)=1$ such that $p(x_{n_k}-x)\leq q_{n_k}$ for all $n_k\in K$. On the other hand, it is well known that a band projection $p_B$ is a lattice homomorphism and it satisfies the inequality $0\leq P_B\leq I$. Hence, we have 
$$
p(P_B(x_{n_k})-P_B(x))=p(|P_B(x_{n_k})-P_B(x)|)=p(P_B|x_{n_k}-x|)\leq p(|x_{n_k}-x|)\leq q_{n_k}
$$
for every $n_k\in K$. Therefore, we get $P_B(x_n)\stpc P_B(x)$.
\end{proof}

We finish this section with statistical convergence on the Dedekind completion Riesz spaces. Recall that a Dedekind complete Riesz space $E^\delta$ is said to be a Dedekind completion of the Riesz space $E$ whenever $E$ is Riesz isomorphic to a majorizing order dense Riesz subspace of $E^\delta$. It is well known that every Archimedean Riesz space has a Dedekind completion (cf. \cite[Thm.2.24]{ABPO}).
\begin{theorem}
	Let $(X,p,E)$ be an $LNRS$. Then, for a Riesz norm  $p^{\delta}_L:X^{\delta}\rightarrow E^{\delta}$ denoted by 
	$p^{\delta}_L(z)=\sup\limits_{0\leq x\leq |z|}p(x)$ for every $z\in X^\delta$ and for a sequence $(x_n)$ in $X$, $x_n\stpc x$ in $(X,p,E)$ iff $x_n\stpc x$ in $(X^\delta,p^\delta,E^\delta)$.
\end{theorem}

\begin{proof}
	Suppose that $x_n\stpc0$ in $(X,p,E)$. Then there exists a sequence $q_n\stpd0$ in $E$ with a set $\delta(K)=1$ such that $p(x_{n_k}-x)\leq q_{n_k}$ for every $n_k\in K$. Since $q_{n_k}\downarrow 0$ in $E$, we have $q_{n_k}\downarrow 0$ in $E^\delta$ (cf. \cite{AS}). Hence, it follows that $p^\delta(x_{n_k}-x)\leq q_{n_k}$ in $E^{\delta}$, and so, we obtain $x_n\stpc x$ in $(X^\delta,p^\delta,E^\delta)$.
	
	Conversely, assume that $x_n\stpc 0$ in $(X^\delta,p^\delta,E^\delta)$. Then there exists a sequence $q_n\stpd0$ in $E^\delta$ with a set $\delta(K)=1$ such that $p^\delta(x_{n_k}-x)\leq q_{n_k}$ for every $n_k\in K$. On the other hand, we have $q_n\stpd0$ in $E^\delta$ (cf. \cite{AS}).  Hence $x_n\stpc 0$ in $(X,p,E)$. 
	
\end{proof}

\textit{Declarations}: This article does not contain any studies with human participants or animals performed by any of the authors.\\
\textit{Data availability statement}: Data sharing not applicable to this article as no datasets were generated or analyzed during the current study.
{\tiny 

}
\end{document}